\documentclass{article}

\usepackage{booktabs} 
\usepackage{array} 
\usepackage{paralist} 
\usepackage{verbatim} 

\usepackage{sectsty}
\allsectionsfont{\sffamily\mdseries\upshape} 

\usepackage[nottoc,notlof,notlot]{tocbibind} 

\usepackage{amsthm}
\newtheorem{mydef}{Definition}

\newtheorem{prop}{Proposition}
\newtheorem{lemma}{Lemma}
\newtheorem{lem}[lemma]{Lemma}

\newtheorem{thm}{Theorem}
\newtheorem{theorem}[thm]{Theorem}

  \newcommand{\thmtext}{If $A\subseteq T_2(\F_q)$ satisfies $e\in A, A=A^{-1}$, and $|A^3|\leqslant K|A|$, then either
\begin{enumerate}
\item there is an $\F^*$-potent group $H$ such that $|A^2\cap H|\gg K^{O(1)}|A|$, or
\item there is a subgroup $U\leqslant U_2(\F_q)$ such that $U\subseteq A^{O(1)}$ and $\gen A/U$ is abelian.
\end{enumerate}}
\newtheorem*{thm:upp2-Fq}{Theorem~\ref{thm:upp2-Fq}}

\newtheorem*{lem:reduction-for-few-ratios}{Lemma~\ref{lem:reduction-for-few-ratios}}

  \newtheorem*{lem:small-fibers}{Lemma~\ref{lem:small-fibers}}
  \newtheorem*{lem:cases-for-upp2}{Lemma~\ref{lem:cases-for-upp2}}
  \newtheorem*{prop:Wg=Wg'}{Proposition~\ref{prop:Wg=Wg'}}
  
    \newtheorem*{thm:energyt2}{Theorem~\ref{thm:energyt2}}
    \newtheorem*{thm:energyh}{Theorem~\ref{thm:energyh}}

\usepackage{amsmath}
\usepackage{amsfonts}
\usepackage{amssymb}
\usepackage{ amssymb }

\usepackage[backend=bibtex,style=numeric]{biblatex}
\usepackage{biblatex} 
\renewbibmacro{in:}{}

\usepackage{textcomp}

\usepackage{hyperref}

\usepackage{ dsfont }

\newcommand{\F}{\ensuremath{\mathbb{F}}}
\newcommand{\C}{\ensuremath{\mathbb{C}}}

\newcommand{\aff}{\ensuremath{\mathrm{Aff}}}
 
\DeclareMathOperator{\stab}{Stab}

\renewcommand{\epsilon}{\varepsilon}

\usepackage{fixme}
\fxsetup{theme=color,mode=multiuser}
\definecolor{fxtarget}{rgb}{0.8000,0.0000,0.0000}
\FXRegisterAuthor{jw}{ajw}{\color{red}JW}
\FXRegisterAuthor{bm}{abm}{BM}

\usepackage{verbatim}

\makeatletter
\newtheorem*{rep@theorem}{\rep@title}
\newcommand{\newreptheorem}[2]{%
\newenvironment{rep#1}[1]{%
 \def\rep@title{#2 \ref{##1}}%
 \begin{rep@theorem}}%
 {\end{rep@theorem}}}
\makeatother
% ...

\newcommand{\gen}[1]{\ensuremath{\left\langle #1\right\rangle}} 

\usepackage[utf8]{inputenc}

\title{Growth in Some Finite Three-Dimensional Matrix Groups}
\author{Brendan Murphy and James Wheeler}

\newcommand{\cP}{\mathcal{P}}
\usepackage{color, enumerate}

\newcommand{\bpm}{\begin{pmatrix}}
\newcommand{\epm}{\end{pmatrix}}

\begin{document}

\maketitle
\begin{abstract}
    We study the growth of product sets in some finite three-dimensional matrix groups.
    In particular, we prove two results about the group of $2\times 2$ upper triangular matrices over arbitrary finite fields: a product set estimate using techniques from multiplicative combinatorics, and an energy estimate using incidence geometry.
    The energy method gives better quantitative results, but only applies to small sets.
    We also prove an energy result for the Heisenberg group.
\end{abstract}

\section{Introduction}

The study of growth in groups is a major area of research in arithmetic combinatorics.
We will use two different methods to look at growth in the group of upper triangular $2\times 2$ matrices over a finite field $\F_q$, and then use one of these methods to look at the Heisenberg group.

One of the oldest questions in this area is  the Erd\H{o}s–Szemer\'edi sum-product conjecture, which is:  every finite set of the integers grows under either addition or multiplication.
That is, for all $\varepsilon>0$ and all sufficiently large finite sets $A\subseteq\mathbb{Z}$, we have
\[
\max(|A+A|,|A\cdot A|)\geqslant |A|^{2-\varepsilon}.
\]
The same question of when a set grows under either addition or multiplication has been studied over other rings and fields, as well as for maps combining addition and multiplication.

Questions have also been asked about growth in other algebraic structures, especially groups. Growth in groups pertains to, for a finite subset $A$ of a group, questions about the size of sets such as $A^k$ or $(A\cup A^{-1}\cup\{e\})^k$ and indeed how these increase as $k$ increases. 
Alternatively, we seek to classify when a subset $A$ of a group does not grow, meaning that $A$ has small tripling: $|A^3|\leqslant K|A|$.
This is part of a trend in arithmetic combinatorics of relaxing an algebraic property or structure and asking what structure remains, in this case relaxing the closure property of subgroups and instead just requiring we are ``nearly closed'' (in a sense that will be made more clear in the following definition).
This idea was formalised with the definition of a $K$-approximate group due to Tao~\cite{tao2006product}.
\begin{mydef}
Let $G$ be a group and $K\geqslant 1$. A non-empty subset $A\subset G$ is a $K$-approximate subgroup of $G$ if:
\begin{itemize}
    \item  It is symmetric, that is if ${ g\in A}$ then ${ g^{-1}\in A}$, and it contains the identity.
\item There exists a subset ${ X\subset G}$ of cardinality ${ |X|\leqslant K}$ such that $ A\cdot A\subset X\cdot A$.
\end{itemize}\end{mydef}
When $K=1$, an approximate group is a subgroup. If a set has small tripling then $(A\cup A^{-1}\cup\{e\})^2$ is a $K^{O(1)}$-approximate group (by Lemma~\ref{lem:PR-ineq}).
As such, the question of when finite subsets of a group grow is equivalent to classifying approximate groups.

In this paper, we study the growth of subsets of the group upper triangular $2\times 2$ matrices over a finite field $\F_q$, as a model case for growth in the group of $n\times n$ upper triangular matrices and other linear groups.
While matrix groups have only one operation (rather than the two from the Erd\H{o}s–Szemer\'edi conjecture),  matrix multiplication involves both additive and multiplicative operations, which suggests that subsets of matrix groups should typically grow.

The question of classifying approximate groups is central to additive combinatorics, and has been studied extensively, for example by Freiman who classified approximate subgroups of the integers  \cite{Freiman73}  and Green and Ruzsa \cite{GR07} who generalised Freiman's result to all abelian groups.

Some matrix groups that have been studied in detail include $G = SL_2(\F_p)$ (Helfgott  \cite{helfgott2005growth}), $SL_3(\F_p)$ (Helfgott \cite{helfgott2008growth}), and $SL_n(\F_p) $ (Pyber-Szabó~\cite{pyber2010growth} and
Breuillard-Green-Tao ~\cite{breuillard2010linear}). Tointon's book~\cite{Tointonnotes2019} collects many of these ideas together.
Of particular note to our paper are the following two results.
We write $\gen{A}$ for the group generated by $A$. 
\begin{theorem}[Gill and Helfgott, Theorem 1 \cite{Gilhelf10}]
\label{thm:GH-soluble}
 Let $A$ be a subset of $GL_n(\F_p)$ such that $\gen{A}$ is solvable.
Then, for every $K \geqslant 1$, either
\begin{itemize}
 \item  $|A^3|\geqslant K|A|$, or else
 \item there is a unipotent subgroup $U_R$, a solvable group $S$, and an integer $k$ depending only on $n$ such that
 \begin{itemize}
    \item $U_R$ and $S$ are both normal in $\gen{A}$, and $S/U_R$ is nilpotent,
    \item $A^k$ contains $U_R$, and
    \item $|A^k \cap S| \geqslant K^{-O_n(1)}|A|$.
 \end{itemize}
\end{itemize}
\end{theorem}
Let $T_n(\F)$ denote the group of invertible $n\times n$ upper triangular matrices over a field $\F$.
Theorem~\ref{thm:GH-soluble} is proved by reducing to the case where $A\subseteq T_n(\F_p)$.
Breuillard and Green \cite{breuillard2011approximate-b} proved a similar result over the complex numbers. 
\begin{theorem}[Breuillard and Green, Theorem 1.4' \cite{breuillard2011approximate-b}]
\label{thm:brgr}
 Let $K > 1$. Suppose that $A \subseteq T_n(\C)$ is a set with
$|A^3| \leqslant K|A|$.
Then there is some set $A' \subseteq A$ with $|A'| > K^{-C}|A|$ that is contained in
a left coset of a nilpotent subgroup of $T_n
(\C)$ of step at most $n - 1$.
\end{theorem}
 
Gill and Helfgott conjectured that their result should hold for any finite field $\F_q$, where $q=p^r$ is any prime power.
By embedding $GL_n(\F_{p^r})$ into $GL_{rn}(\F_p)$, we could apply Theorem~\ref{thm:GH-soluble} to subsets of $GL_n(\F_{p^r})$, but then the constants would tend rapidly to infinite with $r$.
In this paper, we will classify the approximate subgroups of the group $T_2(\F_q)$ of $2\times 2$ upper triangular matrices over an arbitrary finite field $\F_q$. This shows that Theorem~\ref{thm:GH-soluble} holds when $\F_p$ is replaced by $\F_q$ in the simplest case of dimension $n=2$.

Let $U_n(\F_q)$ denote the unipotent subgroup of $T_n(\F_q)$, comprising matrices with 1's on the diagonal.
We say that a matrix $g$ in $T_n(\F_q)$ is \emph{$\F^*$-potent}\/ if $g=\lambda u$ where $\lambda\in\F^*$ and $u\in U_n(\F_q)$.
The collection of $\F^*$-potent matrices is a subgroup of $T_n(\F_q)$, and for $n=2$ it is abelian.
Our first theorem is the following.
\begin{thm}
  \label{thm:2D}
  \label{thm:upp2-Fq}
If $A\subseteq T_2(\F_q)$ satisfies $e\in A, A=A^{-1}$, and $|A^3|\leqslant K|A|$, then either
\begin{enumerate}
\item there is an $\F^*$-potent group $H$ such that $|A^2\cap H|\gg K^{O(1)}|A|$, or
\item there is a subgroup $U\leqslant U_2(\F_q)$ such that $U\subseteq A^{O(1)}$ and $\gen A/U$ is abelian.
\end{enumerate}
\end{thm}
Thus if a subset $A$ of $T_n(\F_q)$ does not grow, then either $A$ is nearly contained in a coset of an abelian subgroup, or the upper right hand entries of the elements of $A$ essentially form a vector space over the subfield generated by ratios of the diagonal of elements of $A$.
Expanding upon this second obstruction, examples 1-approximate subgroups of $T_2(\F_q)$ include the affine subgroup and the upper triangular matrices over a subfield of $\F_q$.
These obstructions are not encountered in the papers of Breuillard and Green~\cite{breuillard2011approximate-b} and Gill and Helfgott~\cite{Gilhelf10} (as $\F_p$ and $\C$ do not have non-trivial finite subfields), and this is the novelties of our proof appear.

Theorem~\ref{thm:upp2-Fq} can be seen as a step towards extending Theorem~\ref{thm:GH-soluble} (and Theorem~\ref{thm:brgr})
to any finite field.
In fact, we can write our conclusions in the same style as Gill and Helfgott's Theorem~\ref{thm:GH-soluble}: in case 1, $U_R=\{I\}$ is trivial and $S$ is the subgroup generated by $A^2\cap H$, and in case 2, we take $U_R=U$ and $S = \gen{A}$.
If we were to consider Theorem~\ref{thm:2D} over the complex numbers rather than $\F_q$, the unipotent subgroup $U$ in case 2 would be trivial, and Theorem~\ref{thm:2D} would have the same conclusion as the $2\times 2$ case of Breuillard and Green's Theorem~\ref{thm:brgr}.



Now we change focus and describe our multiplicative energy results for the upper triangular $2\times2$ matrices and the Heisenberg group over $\F_q$.
Multiplicative energy is a $L^2$ measure of multiplicative structure; specifically, the multiplicative energy (or simply ``energy'') of a finite subset $A$ of a group is given by
$$
E(A) := | \{(g,h,u,v)\in A^4:\, g^{-1} h= u^{-1} v\}|\,=\sum_xr^2_{A^{-1}A}(x)$$
where $r_{A^{-1}A}(x)$ is a representation function  that counts the number of ways $x$  can be expressed as a product $a^{-1}a'\in A^{-1}A$. A standard application of the Cauchy-Schwarz inequality yields the following bounds connecting $E(A)$ with $A^{-1}A$.
\begin{equation} 
E(A) \geqslant \frac{|A|^4}{|A^{-1}A|}
\end{equation}
Thus if the product set $|A^{-1}A|$ is small, the energy of $A$ is large, and vice versa.
In particular, proving an upper bound for the multiplicative energy of a set $A$ is stronger than proving that $A$ grows under multiplication.

Results such as Theorem~\ref{thm:GH-soluble} are non-trivial when $K=|A|^{\delta}$ for some small constant $\delta>0$ depending on the group containing $A$.
The aim of our next set of results is to seek explicit quantitative results, meaning that we would like to find an explicit value $\delta$ for which $K=|A|^\delta$ is non-trivial.
Recent progress made in this direction includes a $\delta=1/20$ growth rate for $SL_2(\F_p)$ due to Rudnev and Skredov~\cite{rudnev2018growth} building upon upon Helfgott's reuslt~\cite{helfgott2005growth}, as well as  Shkredov's work on the Heisenberg group~\cite{shkredov2019remarks} and Dona's work~\cite{dona2019number} on the affine group over $\F_q$.

In the following we write $\Lambda$ for the diagonal subgroup with equal entries on the main diagonal and $T$ for a maximal torus in the affine group.
Typical elements of $\Lambda U_2$ and $\Lambda T$ look like
$$ 
\bpm a & b\\0& a \epm \in \Lambda U_2\,\qquad  \bpm a &(c-a)x\\0& c \epm \in \Lambda T\,.
$$

\begin{theorem}\label{thm:energyt2}
Let $A\subseteq T_2(\F_q)$ and $M_1$ be the maximum number of elements of $A$ in a coset of $\Lambda T$, $M_2$ be the maximum number of elements of $A$ in a coset of $\Lambda U_2$, $M_3$ in a coset of $U_2$, suppose 
$|A|M_3\leqslant p^2$.

Then we have the following energy estimate
\[
E(A) \lesssim |A|^{5/2} M_2^{1/2} + |A|^2 M_1,
\]
and hence\footnote{$X \gg Y$ and $Y \ll X$ indicates that there is an absolute constant $C>0$ such that $X \geqslant CY$, and the notations $\lesssim, \gtrsim$ hide, on top of this, powers of $\log |A|$.}
$$
|A^{-1}A|,\,|AA| \gtrsim \frac{ |A|^{2} } { M_1+ \sqrt{|A| M_2}}.
$$
\end{theorem}

We prove a similar result for the Heisenberg Group, which is the group of matrices of the form 
\[
H=H(\F_q):=\left\{\bpm 1& g_1 & g_3\\0&1&g_2\\0&0&1\epm:g_1,g_2,g_3\in\F_q\right\}.
\]
We write $LZ$ for the two-dimensional abelian subgroups of the Heisenberg group where $L$ is the subgroups defined by $g_3=0$ and $\alpha g_1+\beta g_2=0$ for some $\alpha,\beta$. and $Z$ is subgroup formed of elements of the form $(0,0,g_3)$, the centre of $H$. 

\begin{theorem}\label{thm:energyh}
Let $A\subset H(\F_q)$, let $m$ be the maximum number of elements of $A$ in a coset of $Z$, and let $M$ be the maximum number of elements in a coset of $LZ$.
If $|A|m\leqslant p^2$ and $m\leqslant \sqrt{|A|}$, then we have the energy estimate 
\[
E(A) \lesssim |A|^{5/2} m + |A|^2 M,
\]
and hence,
$$
|A^{-1}A|,\,|AA| \geqslant \frac{ |A|^{2} } { M+ m\sqrt{|A|} }\,.
$$
\end{theorem}
This generalises a theorem of Shkredov~\cite[Theorem 2]{shkredov2019remarks}. 

The Heisenberg group interests us for a few reasons.
First it is a step towards generalising Theorem~\ref{thm:upp2-Fq} to $T_3(\F_q)$, but is still three dimensional in the same was $T_2(\F_q)$ is and thus the energy approach is very similar.
It is also interesting as it contains an example of a subset of upper triangular matrices that has a set which does not grow and is not mostly contained in a coset of an abelian group.
In particular consider the set \begin{equation}
A:=\left\{\bpm 1& g_1 & g_3\\0&1&g_2\\0&0&1\epm:g_1,g_2\in [1,\dots ,n] g_3\in [1,\dots n^2]\right\}.\label{example1}\end{equation}
The set $A$ neither grows nor is it a coset of an abelian subgroup, illustrating why Theorem~\ref{thm:brgr}  concludes that sets in $T_n(\C)$ with small tripling must have large overlap with a coset of a nilpotent group of step $n-1$ (with n being from $n\times n$ matrices).

In the proofs of both of our energy results we will follow ideas of Petridis,  Roche-Newton,  Rudnev and  Warren \cite{petridis2019energy}, who showed the following result for the affine group, $$\aff(\F):=\left\{
  \begin{pmatrix}
    a & b \\
    0 & 1 \\
  \end{pmatrix}
\colon a,b\in\F, a\neq 0
\right\}.$$
\begin{theorem}[Corollary 6 \cite{petridis2019energy}]
Let $A$ be a subset of the affine matrices over $\F_p$ have no more than $M$ elements in a coset of a torus and no more than $m$ elements in a coset of the unipotent subgroup.
Suppose $m|A| \leqslant p^2$.
Then
\[
\min\{|AA|, |A^{-1}A|\} \gg m^{-1/2}|A|^{3/2} + M^{-1}|A|^2.
\]
In particular, if $|AA| = K|A|$ and $m|A|\leq p^2$, then $A$ has $\gg |A|/K^2$ elements in a coset of the unipotent subgroup or $\gg |A|/K$ elements in a coset of a torus.
\end{theorem}

The two approaches used in proving Theorem~\ref{thm:upp2-Fq} and Theorem~\ref{thm:energyt2}, one based on multiplicative combinatorics, the other based on incidences, each have their own advantages.
The former is universal and handles $\F_q$, but it is quantitatively weaker; the latter can be used only in some lower-dimensional groups, but is quantitatively better, asking in particular for no symmetry assumptions, and dealing with $AA$ instead of $AAA$.
The main limitation for Theorem~\ref{thm:energyt2} is that it only works for small sets, and thus cannot pick up the second case of Theorem~\ref{thm:upp2-Fq}; precisely why this happens will be looked at in more detail later in the paper.
Put succinctly, Theorem~\ref{thm:energyt2} is stronger when it applies but Theorem~\ref{thm:upp2-Fq} applies all the time.
We remark that it would be interesting to prove energy results for $SL_2(\F_p)$ and higher dimensional groups.

We will now describe the structure of the rest of the paper.
In Section~\ref{sec:prelim}, we cover some notation as well as some useful lemmas.
Section~\ref{sec:2x2 fq} contains the proof of Theorem~\ref{thm:upp2-Fq}.
Finally, Section~\ref{sec:energy} contains the proofs of our quantitative energy results, first for the upper triangular matrices (Theorem~\ref{thm:energyt2}) and then for the Heisenberg group (Theorem~\ref{thm:energyh}).

\section{Preliminaries}
\label{sec:prelim}
Before we head into the main content of this paper and the proof of Theorem~\ref{thm:upp2-Fq} we will set up our notation, our definitions, and state several useful lemmas from multiplicative combinatorics.

\subsection{Notation and Definitions}
We start by defining our notation for various subgroups, repeating a few definitions made in the introduction so that all of our notation can be found in one place.
Recall that $T_n(\F)$ is the set of invertible upper triangular $n\times n$ matrices with entries in a field $\F$ and that $U_n(\F)$ denotes the subgroup of $T_n(\F)$ consisting of unitriangular matrices (that is, with 1's on the diagonal).
In this paper we will focus on $n=2$ and $\F$ will be a finite field with $q$ elements and characteristic $p>0$.
Let $D_n(\F)$ denote the subgroup of $T_n(\F)$ consisting of diagonal matrices.
We note that there is a short exact sequence
\[
1\to U_n(\F)\to T_n(\F)\xrightarrow{\pi} D_n(\F)\to 1,
\]
and $T_n(\F)\cong D_n(\F)\ltimes U_n(\F)$.
Explicitly the map $\pi:T_n(\F)\to D_{n}(\F)$ takes an upper triangular matrix and returns the diagonal matrix whose diagonal elements match those of the upper triangular matrix.
In two dimensions that is
$$\pi:\begin{pmatrix}
  a&b\\0&c
\end{pmatrix}\mapsto\begin{pmatrix}
  a&0\\0&c
\end{pmatrix}.$$

Two other types of subgroups of $T_2(\F_q)$ we will be particularly interested in are tori and $\F^*$-potent subgroups.
We define a torus, $T$, to be a subgroup of $T_2(\F)$ conjugate to the diagonal subgroup $D_2(\F)$.
That is \[
T=\left\{
  \begin{pmatrix}
    a & (a-d)z \\
    0 & d \\
  \end{pmatrix}
\colon a,d\in\F^*
\right\},
\]
where $z$ is an element in $\F$.
An $\F^*$-potent group is an abelian subgroup of $T_n(\F_q)$ of the form
\[
H=\left\{
  \begin{pmatrix}
    a & b \\
    0 & a \\
  \end{pmatrix}
\colon a\in\F^*, b\in\F
\right\}.
\]



We use the commutator notation $[g,h]:=g^{-1}h^{-1}gh$ and $[G,H]$ is the subgroup generated by all commutators of the form $[g,h]$ such that $g\in G, h\in H$.
If we are to consider all the conjugates of an element $x$ by a set $A$ we will use the following notation, $x^A:=\{a^{-1}xa\colon a\in A\}$.



\subsection{Multiplicative combinatorics}

If $A$ is a subset of a multiplicative group, we use $A^n$ to denote the set of $n$-fold products of elements of $A$, and we define $A_{(n)}:=(A\cup A^{-1}\cup\{e\})^n$ where $e$ is the identity.
This is useful when we want to be like an approximate group, that is we want to ensure we have inverses and the identity.

The following lemma links small tripling ($|A^3|\ll|A|$) with the size of $A_{(k)}$.
\begin{lem}[Pl\"{u}nnecke-Ruzsa]
  \label{lem:PR-ineq}
Suppose that $A$ is a finite subset of a group such that $|A^3|\leqslant K|A|$.
Then
\[
|A_{(3)}|\leqslant 27K^3|A|,
\]
for all $k\geqslant 3$
\[
|A_{(k)}|\leqslant \left(\frac{|A_{(3)}|}{|A_{(1)}|}\right)^{k-2}|A_{(1)}|,
\]
and hence, for all $k\geqslant 3$,
\[
|A_{(k)}|\leqslant 27^k K^{3(k-2)}|A|.
\]
\end{lem}
This lemma is obtained from repeated applications of the Ruzsa triangle inequality and more details can be found in Tao's paper~\cite[Lemma 3.4]{tao2008product}.

The next lemma is the orbit-stabiliser theorem for sets, one of many results from group theory that can be adapted for approximate groups \cite[Lemma 4.1]{helfgott2015growth}.
Recall that if group $G$ acts on a set $X$, then the  stabiliser of an element $x$ of $X$, is the subgroup $\stab(x)$ of $G$ consisting of all elements that fix $x$.
\begin{lem}[Orbit-Stabiliser Theorem for sets]
  \label{lem:orb-stab}
Suppose the group $G$ acts on a set $X$, $x\in X$, and $A\subseteq G$ is finite.
Then there exists $a_0$ in $A$ such that
\begin{equation}
  \label{eq:orb-stab-lower}
  |(a_0^{-1}A)\cap\stab(x)|\geqslant\frac{|A|}{|A(x)|},
\end{equation}
and for all finite sets $B\subseteq G$,
\begin{equation}
  \label{eq:orb-stab-upper}
  |AB|\geqslant |\stab(x)\cap B||A(x)|.
\end{equation}
\end{lem}
We often specialise Lemma~\ref{lem:orb-stab} to the action of a group $G$ on a subgroup $H$ by left multiplication, so that the stabiliser of $H\in G/H$ is $H$ itself and the orbit of $H$ under a set $A$ is $AH/H$, which is the number of distinct coset representatives in $A$.
If $\psi\colon G\to G/H$ is the quotient map, then $A(H)=AH/H=\pi(A)$.
The subgroup $H$ does not need to be normal.

We use the following lemma when we wish to move from growth in a group to growth in a subgroup.
\begin{lem}
\label{lem:growth-in-subgroups}
Suppose that $H$ is a subgroup of $G$ and that $A\subseteq G$ satisfies $|A^3|\leqslant K|A|$.
If $B:=A^{-1}A\cap H$, then
\[
|B^k|\leqslant |A_{(2k)}\cap H| \ll K^{6k-3}|B|.
\]
\end{lem}
\begin{proof}
By Lemma~\ref{lem:orb-stab} Equation~(\ref{eq:orb-stab-lower}) where we have specialised to the action of a group $G$ on a subgroup $H$ by left multiplication, we have $|B|\geqslant |A|/|A(H)|$ (where $A(H)$ number of distinct cosets of $H$ determined by elements of $A$).
On the other hand, $B^k \subseteq (A^{-1}A)^k\cap H\subseteq A_{(2k)}\cap H$.

So using Lemma \ref{lem:PR-ineq} we have that $$|B^k| \leqslant |(A^{-1}A)^k\cap H|\leqslant |A_{(2k)}\cap H|.$$

By Lemma~\ref{lem:orb-stab} Equation~(\ref{eq:orb-stab-upper}), taking $A$ as $A$, and $B$ as $A_{(2k)}$ we also have
$$|AA_{(2k)}|\geqslant |A_{(2k)}\cap H||A(H)|.$$

So, combining the above and using Lemma~\ref{lem:PR-ineq} to pass from $|A_{(2k)}|$ to $|A|$, we get that $$|B^k|\leqslant |A_{(2k)}\cap H| \leqslant |AA_{(2k)}|/|A(H)|\ll K^{6k-3}|A|/|A(H)|\leqslant K^{6k-3} |B|.$$
\end{proof}

Given a function $f\colon A\to B$, we say that $\phi\colon B\to A$ is a \emph{right inverse}\/ if $f\circ \phi(x)=x$ for all $x\in B$.
\begin{lem}
  \label{lem:covering-by-kernel}
Let $N$ be a normal subgroup of $G$, let $\pi\colon G\to G/N$ be the quotient map, and let $A$ be a finite subset of $G$.
Let $k$ be a positive integer and let $\phi\colon\pi(A^k)\to A^k$ be a right inverse.
Then for all $a$ in $A^k$, we have
\[
a\in\phi(\pi(a))(A^{-k}A^k\cap N).
\]
Hence
\[
A^k\subseteq \phi(\pi(A^k))(A^{-k}A^k\cap N).
\]
\end{lem}
This lemma is part iv of Lemma 2.12 (Tao's splitting Lemma) in Tointon's paper~\cite{tointon2012freimans}, which itself follows Tao~\cite{tao2006product}.
It essentially comes from seeing that the identity is contained in $A^{-k}A^k\cap N$ and $a\in\phi(\pi(a))$.

The proof of Theorem~\ref{thm:upp2-Fq} also requires the following sum-product result \cite[Theorem D]{murphy2017products}.
\begin{prop}
  \label{prop:Fq-sum-product}
Let $X$ be a finite subset of an $\F_q$-vector space, let $D\subseteq\F_q$ be a set of scalars, and let $F=\gen D$ be the subfield generated by $D$.
If $|X+DX|\leqslant K|X|$ for some $K\geqslant 1$, then either $K\geqslant |D|^{1/10}$ or
\[
|X|\geqslant \frac 1{2K^4}|\mathrm{Span}_F(X)|,
\]
and
\[
\mathrm{Span}_F(X)\subseteq 4DX-4DX.
\]
\end{prop}
Thus a finite subset of a vector space which doesn't grow under taking the sum with a dilate of our subset, then either the set of dilates is small or our subset is a large proportion of the vector space it spans over the field generated by the dilates. 

The last part of Proposition~\ref{prop:Fq-sum-product} is not stated explicitly in the reference \cite{murphy2017products}, but follows from arguments in that paper's Section 4.4: there exists an element $\xi_0=(x_1-x_2)/(s_1-s_2)$ with the $x_i$'s in $X$ and the $d_i$'s in $D$ such that $|X+D\xi_0|>\frac 12 |W|$.
Thus $(d_1-d_2)X+D(x_1-x_2)$ is a subset of $W$ with density greater than $1/2$, so the statement follows from a version of the Cauchy-Davenport theorem \cite[Lemma 2.1]{helfgott2008growth}.

\section{Product Theorem in \texorpdfstring{$T_2(\F_q)$}{T2(Fq)}}
\label{sec:2x2 fq}

In this section we will prove our first main result (Theorem~\ref{thm:upp2-Fq}), which we restate before we embark on its proof.
\begin{thm:upp2-Fq}
  \thmtext
\end{thm:upp2-Fq} 
The unipotent subgroup $U$ has the form
\[
U=\left\{
  \begin{pmatrix}
    1 & w \\
    0 & 1 \\
  \end{pmatrix}
\colon w\in W
\right\},
\]
where $W$ is a vector space over a subfield $F\leqslant \F_q$ and is thus a subgroup of the full unipotent subgroup.

We note that this result can be compared with Theorem 35 of Murphy~\cite{BM17}, which looks at the affine case.

\begin{proof}[Proof of Theorem~\ref{thm:2D}]
First, note that
\[
[A,A]=\{[a,a']\colon a,a'\in A\}\subseteq A^4\cap U_2(\F).
\]
Let $u\colon\F_q\to U_2(\F_q)$ be the isomorphism defined by
\[
u\colon x\mapsto
\begin{pmatrix}
1 & x \\
0 & 1 \\
\end{pmatrix}
\]
and let $X:=u^{-1}(A^4\cap U_2(\F_q))$.

Let $\chi\colon T_2(\F_q)\to \F_q^*$ denote the homomorphism defined by
\[
\chi\colon
\begin{pmatrix}
a_{11}&a_{12}\\
0&a_{22}\\
  \end{pmatrix}
  \to a_{11}/a_{22},
\]
and let $D=\chi(A)$.
Since conjugation of $U_2(\F_q)$ by an element $g$ of $T_2(\F_q)$ corresponds to multiplication by $\chi(g)$, we have $u(X+DX)\subseteq A^{10}\cap U_2(\F_q)$.

By Lemma~\ref{lem:growth-in-subgroups},
\[
|X+DX|=|u(X+DX)|\leqslant |A^{10}\cap U_2(\F_q)|\ll K^{30}|A^2\cap U_2(\F_q)|\leqslant K^{30}|X|,
\]
hence by Proposition~\ref{prop:Fq-sum-product}, either $|D|\ll K^C$ or there is a vector space $W$ over the field $F$ generated by $D$ such that $X\subseteq W\subseteq 6DX-6DX$.

In the first case, there is a $\F^*$-potent subset $S$ of $A^2$ such that $|S|\gg K^{-C}|A|$.

In the second case, set $U=u(W)$.
Then $U$ is contained in $A^C\cap U_2(\F_q)$ and $A^2\cap U_2(\F_q)\subseteq U$, so for any set $\Lambda\subseteq A$ of left coset representatives of $A$ modulo $U_2(\F_q)$, we have
\[
A\subseteq \Lambda\cdot U
\]
by Lemma~\ref{lem:covering-by-kernel}.
Since the ratios of the diagonal terms of elements of $\Lambda$ are contained in $D$ and $W$ is a vector space over the field $F$ generated by $D$, the subgroup $U$ is normalised by $\Lambda$, hence $\gen A =\gen{\Lambda}U$.

It remains to show that $\gen A/U$ is abelian.
Let $\phi\colon\gen A\to\gen A/U$ denote the quotient map.
Since $\Lambda\subseteq A$, we have $[\Lambda,\Lambda]\subseteq U$, so
\[
[\phi(\Lambda),\phi(\Lambda)] = \phi([\Lambda,\Lambda])=\{e\},
\]
so $\phi(\Lambda)$ generates an abelian subgroup of $\gen A/U$.
But 
\begin{equation}
  \label{eq:1}
  \gen{\phi(\Lambda)}\cong \gen{\Lambda}U/U=\gen A/U,
\end{equation}
so $\gen A/U$ is abelian, as claimed.
\end{proof}

\section{Energy estimates}
\label{sec:energy}

In this section we will prove Theorems~\ref{thm:energyt2} and~\ref{thm:energyh}. We will follow the ideas from Petridis et al.'s paper~\cite{petridis2019energy} closely. 

\subsection{Energy in \texorpdfstring{$2\times 2$}{2x2} Triangular Matrices}\label{secsub:energyT2}

We start by noting that $T_2(\F_q)$ is the direct product of its centre $\Lambda$, which consists of the diagonal matrices with equal elements on the main diagonal, and a subgroup $\Gamma$ that is isomorphic to the affine subgroup. The projection $\rho\colon T_2(\F_q)\to \Gamma$ is the homomorphism given by
$$
\rho\colon\bpm a & b\\0& c \epm \mapsto \bpm a/c & b/c\\0& 1 \epm\,.
$$
We also note that $\Gamma$ contains the normal unipotent subgroup $U_2(\F_q)$.
For $g\in A$, we call a multiple of $g$ by an element of $\Lambda$ a  {\em dilate} of $g$.

As stated above, the multiplicative energy $E(A)$ of a subset $A$ of a group is defined by
$$
E(A) := | \{(g,h,u,v)\in A^4\colon g^{-1} h= u^{-1} v\}|,
$$
and a standard application of the Cauchy-Schwarz inequality yields the following bound connecting $E(A)$ with $A^{-1}A$.
\begin{equation} \label{CSbasic}
E(A) \geqslant \frac{|A|^4}{|A^{-1}A|}.
\end{equation}
Another nature definition for the multiplicative energy of $A$ is
\[
E^*(A)  :=|\{(g,h,u,v)\in A^4\colon g h= u  v\}|,
\]
which satisfies
\[
E^*(A) \geqslant \frac{|A|^4}{|AA|}.
\]
Shkredov \cite[Section 4]{shkredov2019modular} proved the inequality $E^*(A)\leqslant E(A)$, and so we focus on bounding $E(A)$.

We cannot expect to have a non-trivial upper bound on $E(A)$ unconditionally, since $A$ can lie in a coset of an abelian subgroup.
The maximal abelian subgroups of $T_2(\F_q)$ arise as $\Lambda$ times a maximal abelian subgroup of $\Gamma$.
In particular,
$$ 
\bpm a & b\\0& a \epm \in \Lambda U_2\,\qquad  \bpm a &(c-a)x\\0& c \epm \in \Lambda T\,,
$$
where $T$ is a maximal torus in $\Gamma$.

We use Rudnev's point-plane incidence bound~\cite{Rudnev2018}.
\begin{theorem}
\label{thm:pointplane}
 Let $\F_q$ be a field, and let $P$ and $\Pi$ be finite sets of points and planes
respectively in $\mathbb{P}^3$.
Suppose that $|P| \leqslant |\Pi|$, and that $|P| \ll p^2$.
Let $k$ be the maximum number of collinear points in $P$. Then the number of incidences
satisfies
$$I(P, \Pi) \ll |\Pi||P|^{1/2} + k|\Pi|.$$
\end{theorem}

We note that although this is a result over $\F_q$ we are still restricted in terms of $p$ and need small sets (about at most $p$).
As we will explain after the proof this is why we have some differences with Theorem~\ref{thm:upp2-Fq}. We will be following the ideas from \cite{petridis2019energy} in the following.

Before proving Theorem~\ref{thm:energyt2}, we restate it for the reader's convenience.
\begin{thm:energyt2}
Let $A\subseteq T_2(\F_q)$ and $M_1$ be the maximum number of elements of $A$ in a coset of $\Lambda T$, $M_2$ be the maximum number of elements of $A$ in a coset of $\Lambda U_2$, $M_3$ in a coset of $U_2$, suppose 
$|A|M_3\leqslant p^2$.

Then we have the following energy estimate
\[
E(A) \lesssim |A|^{5/2} M_2^{1/2} + |A|^2 M_1,
\]
and hence
$$
|A^{-1}A|,\,|AA| \gtrsim \frac{ |A|^{2} } { M_1+ \sqrt{|A| M_2}}.
$$

\end{thm:energyt2}
\begin{proof}
We begin by splitting $A$ into dyadic pieces, then we estimate the energy of each dyadic piece by partitioning them further and controlling these with the point-plane theorem.
 
 Define $A_m\subseteq A$ by
\[
A_m:=\{g\in A\colon m\leq|A\cap g\Lambda|<2m\}.
\]
Then
\[
A=\bigcup_{j=0}^M A_{2^j},
\]
where $M\ll \log |A|$.
Thus by two applications of Cauchy-Schwarz, we have
\begin{equation}
    \label{eq:energy-of-union}
    E(A) = E(\cup_{j=0}^M A_{2^j}) \lesssim \sum_{j=0}^M E(A_{2^j}).
\end{equation}

We will show that for any $m\geq 1$, we have
\begin{equation}
\label{eq:energy-bound-for-dyadic-piece}
    E(A_m)\ll |A_m|^{5/2}M_2^{1/2}+M_1|A_m|^2,
\end{equation}
provided that $M_3|A_m|\ll p^2$.
Since $M_3|A|\ll p^2$ by assumption, we have $E(A_{2^j})\ll |A|^{5/2}M_2^{1/2}+M_1|A|^2$ for all $j$, so we may bound the right-hand side of equation~\ref{eq:energy-of-union} by taking a supremum:
\[
E(A)\lesssim|A|^{5/2}M_2^{1/2}+M_1|A|^2.
\]
Thus the proof is complete, pending the proof of equation~\ref{eq:energy-bound-for-dyadic-piece}.

We write
\begin{equation}
Q(A_m):=\{(g,h,u,v) \in A_m^4 :g^{-1} h = u^{-1}v\},
\label{e:quad}\end{equation}
so that $E(A_m) = |Q(A_m)|$.
We will write elements of $T_2(\F_q)$ as
$$ 
g=\bpm g_1 & g_2\\0& g_3 \epm,
$$
adopting the convention that if $h\in T_2(\F_q)$, then $h_1,h_2,h_3$ bear the same relationship to $h$ as $g_1,g_2,g_3$ do to $g$.

To proceed, we partition the set $Q(A_m)$ into pieces, which we will then be able to control them via the point-plane incidence bound (Theorem~\ref{thm:pointplane}).

We can split the energy of $A$ up as follows
\[
E(A_m)=\sum_{C_1,C_3} |\{(g,h,u,v) \in Q(A_m)\colon  g_1v_1=C_1=h_1u_1, g_3v_3=C_3=h_3u_3\}|.
\] 
To see that this is correct, suppose that $g^{-1} h = u^{-1} v$.
Then, by comparing the three entries in the matrices, we have a set of three equations:
\begin{equation}\label{e:coord}
g_1v_1 = h_1u_1,\qquad g_3v_3= h_3u_3,\qquad u_1h_2 - u_1h_3\frac{g_2}{g_3}  =g_1v_2-g_1v_3\frac{u_2}{u_3}.
\end{equation}
We write $Q_{C_1,C_3}$ for the set of solutions corresponding to the fixed pair of values  $C=(C_1,C_3)$ in the decomposition above:
\[Q_C=Q_{C_1,C_3}:=|\{(g,h,u,v) \in Q(A_m) : g_1v_1=C_1=h_1u_1\,, \; \;g_3v_3=C_3=h_3u_3\}|.
\]
We define additional sets, also indexed by pairs of values $C=(C_1,C_3)$ from the above decomposition:
\[
\cP_C =\cP_{C_1,C_3}: = \{(g,v) \in A_m \times A_m : g_1v_1=C_1,\;\; g_3v_3 = C_3\}.
\]

Next, we seek to bound the quantity $Q_{C}$ using Theorem~\ref{thm:pointplane}.
The last equation in \eqref{e:coord} is the condition that represents point-plane incidences.
The planes are given by projective co-vectors
$$\left(-u_1h_2:u_1h_3:1:-\frac{u_2}{u_3}\right)\,,$$
and points by projective vectors
$$
\left(1:\frac{g_2}{g_3}:g_1v_2:g_1v_3\right) = \left(1:\frac{g_2}{g_3}:g_1v_2: C_3\frac{g_1}{g_3}\right)\,.
$$
The points and planes as above are multisets, since the ratios $\frac{g_2}{g_3}$ and $\frac{g_1}{g_3}$ are defined module $\Lambda$.
By the definition of $A_m$, each point and plane occurs with multiplicity (or ``weight'') $\sim m$.

Thus, to account for the weights, we consider the worst possible case. This is when the number of points/planes is equal to $|\cP_C|/m$ and each incidence is counted $m^2$ times

\begin{equation}
Q_C\ll | \cP_C |^{3/2} \sqrt{m} +  (km) | \cP_C |\,,
\label{ppest}\end{equation}
where $k$ is the maximum number of collinear points and the latter estimate is valid provided that we have the following which is due to the $p$-constraint on the application of the point-plane. 
\begin{equation}\label{ppconstr}
| \cP_C | \leqslant mp^2.
\end{equation}
We remark that if points in $\cP_C$ are collinear, then their projection on the coordinates $(1:\frac{g_2}{g_3}:\frac{g_1}{g_3})$ are collinear, so we get a line in $\Gamma$.
A line in $\Gamma$ is a coset of a torus $T$ or the unipotent group $U_2(\F_q)$.
If the line (coset in $\Gamma$) has $k$ elements and each element has $\sim m$ dilates in $T_2(\F_q)$, the quantity $km$ is bounded by the maximum number of elements of $A$ in a coset of $\Lambda T$ or $\Lambda U_2$, which we bound by $M_1+M_2$.

To continue, we will sum the estimate \eqref{ppest} over all values of $(C_1,C_3)$ using the fact that  
\[
\sum_{(C_1,C_3)}| \cP_{(C_1,C_3)} | = |A_m|^2,
\]
and a supremum estimate for $| \cP_C|$.
Namely, for a fixed $C=(C_1,C_3)$, if we consider the set $\cP_C$, if we know $g\in A_m$ then we know $v_1$ and $v_3$. The maximum number of $v_2$, knowing $v_1$ and $v_3$ is the maximum number $n$ of elements of $A_m$ in a coset of $U_2$.
Since each element has $\sim m $ dilates in $A_m$, then $n\ll M_2/m$.

Hence by summing estimate \eqref{ppest} over all values of $(C_1,C_3)$ and using the above facts we have that
\begin{equation}\label{bound:energy}
E(A_m) \ll |A|^{5/2} M_2^{1/2} + |A|^2 M_1,
\end{equation}
and this is valid, by \eqref{ppconstr} as long as 
\[
| \cP_C|\leqslant  |A_m|n \leqslant m p^2,
\]
this constraint appears with $m=1$ and $M_3$ replacing $n$ in the statement of the theorem and thus we are done. \end{proof}

Both this result and Theorem~\ref{thm:upp2-Fq} say that one obstructure to growth is being close to a coset of one of the following subgroups
$$ 
\bpm a & b\\0& a \epm \in \Lambda U_2\,\qquad  \bpm a &(c-a)x\\0& c \epm \in \Lambda T.
$$
However, because we require that $|A|M_3\leqslant p^2$ in Theorem~\ref{thm:energyt2}, we do not have enough elements to fill up a unipotent subgroup, and so we never encounter the second case of Theorem~\ref{thm:upp2-Fq} when $U$ is not trivial ($U$ being trivial is the torus case here).
This highlights a weakness of this second approach, though it is a much stronger result quantitatively.

 \subsection{Energy in the Heisenberg Group}
 
In this subsection we prove a similar energy bound as for the Heisenberg group, generalising results of Shkredov~\cite{shkredov2019remarks}.
We denote the Heisenberg group by
\[
H=H(\F_q):=\left\{\bpm 1& g_1 & g_3\\0&1&g_2\\0&0&1\epm:g_1,g_2,g_3\in\F_q\right\},
\]
where we write
$$
 g=(g_1,g_2,g_3)\qquad \mbox{or}\qquad g=\bpm 1& g_1 & g_3\\0&1&g_2\\0&0&1\epm, 
$$
for a specific element of $H$. 

The multiplication rule and other useful relations are as follows:
 $$gh=(g_1+h_1,g_2+h_2,g_3+h_3+g_1h_2)$$
 $$g^{-1}=(-g_1,-g_2,-g_3+g_1g_2),\,$$
 $$g^{-1}\cdot h = (h_1-g_1,h_2-g_2, (h_3-g_3)+g_1(g_2-h_2)\,.
 $$
We also note that $(0,0,0)$ is the identity element.

 Moving on to the proof, we will follow very much the same ideas as in the proof of Theorem~\ref{thm:energyt2}.
 
Recall the statement we wish to prove and that $L$ is a linear subspace in the $(g_1,g_2)$ variables and $Z$ is the centre of $H$.
\begin{thm:energyh} Let $A\subset H(\F_q),$ with $|A|m\leqslant p^2$, where $m$ is the maximum number of elements in  a coset of $Z$. If $M$ is the maximum number of elements in  a coset of $LZ$, then  for $m\leqslant \sqrt{|A|}$ we have the estimate $$E(A) \ll |A|^{5/2} m + |A|^2 M,$$ and hence,
$$
|A^{-1}A|,\,|AA| \geqslant \frac{ |A|^{2} } { M+ m\sqrt{|A|} }\,,
$$
\end{thm:energyh}
We note that unlike the $T_2(\F_q)$ case we do not lose any log factors in the bound for the Heisenberg group.

 \begin{proof}
 We will reuse the notation 
\begin{equation}
Q(A):=\{(g,h,u,v) \in A^4 :g^{-1} \cdot h = u^{-1} \cdot v\}.
\end{equation}
 
 Equating $ g^{-1}\cdot h=u^{-1}\cdot v$, and fixing
 $$
 h_1+u_1 = C_1 = g_1+v_1,\qquad h_2+u_2 = C_2 = g_2+v_2\,,
 $$
 then the third equation is 
 \begin{equation}
 [-(g_3+v_3) + g_1g_2 -C_2g_1]+g_1u_2-g_2u_1 + [h_3+u_3 - u_1u_2+C_2u_1]=0\,.
 \label{e:coord2}\end{equation}
 As before this is a weighted point-plane equation.  Following the same general method, let $Q_{C_1,C_2}$ denote the set of solutions corresponding to the fixed pair of values  $C=(C_1,C_2)$ in the decomposition above, that is,
\[Q_C=Q_{C_1,C_2}:= \left|\left\{(g,h,u,v) \in Q(A) : \begin{array}{l}g_1+v_1=C_1=h_1+u_1\,\\  g_2+v_2=C_2=h_2+u_2\end{array}\,\right\}\right|.\]
 
 Still following the previous subsections notation, only now $C=(C_1,C_2)$, we define the set 
\[
\cP_C =P_{C_1,C_2}: = \{(g,v) \in A \times A : g_1+v_1=C_1,\;\; g_2+v_2 = C_2\}.
\]
 As before we are going to bound the quantity $Q_{C}$ using Theorem~\ref{thm:pointplane}. Equation \eqref{e:coord2} represents point-plane incidences, where points are given by projective vectors and planes by projective co-vector
\begin{equation}
(-(g_3+v_3) + g_1g_2 -C_2g_1: g_1:g_2:1)\,,
\qquad
 (1: u_2:-u_1:h_3+u_3 - u_1u_2+C_2u_1)\,.\label{hpoint}
\end{equation}

 The points and planes come with multiplicity as they did in the last case. This multiplicity, in the case of points, is the number of realisations of the sum $g_3+v_3$. Observe that given $C_1,C_2$ and knowing $g_1,g_2$ we then known $v_1,v_2$.  Thus the maximum number of realisations of the sum $g_3+v_3$ (since $g_1,g_2$, as well $v_1,v_2$ are fixed for a given point  is bounded by the maximum number of elements of $A$ in a coset of $U$, which we denote as $m$.

We continue by applying the point-plane bound,
as in the previous section, only now $C=(C_1,C_2)$.
Again, to account for weights, we consider the worst possible case, when the number of points/planes is equal to $|\cP_C|/m$, each incidence is counted $m^2$ times. Summing over $C$ we obtain, once again, the estimate \eqref{ppest}. For the $p$-constraint on the application of the point-plane theorem we take the most ample $|\cP_C|\leqslant p^2$, i.e. the case $m=1$.

Observe that geometrically, for a given value of $(g_3+v_3)$, \eqref{hpoint} is a quadric over the $(g_1,g_2)$ plane, and hence $k$ is bounded by the number of collinear points in the $(g_1,g_2)$ plane, and the $km$ term can be interpreted as the maximum number of elements of $A$ in a coset of a subgroup $H=LZ$.

Summing over $C$, since $|\cP|\leqslant |A|m$, we get the analogue of \eqref{bound:energy} as follows:
 
\begin{equation}\label{bound:energy:h}
E(A) \ll |A|^{5/2} m + |A|^2 M\,,
\end{equation}
where $M$ is the maximum number of elements in a coset of $LZ$. Hence, we have established the theorem.
\end{proof}

Note that the result is sharp in the sense that if $m\gg \sqrt{|A|}$ we can have $E(A)\gg |A|^3$.  Indeed, take the example given in Equation~\ref{example1} from the introduction we he have $A=\{(g_1,g_2,g_3)\} \cong [1,\ldots n]\times [1,\ldots n] \times [1,\ldots n^2].$ Then, if, say $g\cdot h=s$, this means
$$
s_1 = g_1+h_1,\qquad s_2=g_2+h_2,\qquad s_3 = g_3+h_3+ g_1h_2\,,
$$
so a typical $s\in AA$ will have $\sim n\times n\times n^2=|A|$ representations as a product of two elements, and $|AA|\sim |A|$. Moreover, further multiplication by elements of $A$ will not cause growth either.

We remark that a stronger (for smaller values of $m$) bound $E(A)$, which would yield $|AA|\gg|A|^{7/4}$ if $m=1$ was recently obtained by Shkredov \cite[Theorem 13]{shkredov2019remarks} who estimated $E(A)$ in the special  Cartesian product case, namely when each component of $g=(g_1,g_2,g_3)$ lies, independently, in some scalar set. This was done by rather similar methods, but  the Cartesian product setting enables one to apply an incidence bound twice, rather than once. In particular observe that if $m=1$, that is $A$ is a graph over a set $B$ in the $(g_1,g_2)$-variables, then $\sum_C|\cP_C|^2$ equals the additive energy of $B$.



\printbibliography {}

\end{document}